\documentclass[11pt]{article}
\usepackage{amsmath,amsthm,verbatim,amssymb,amsfonts,amscd, graphicx}
\usepackage{pdfsync}
\usepackage{parskip}%create a blank line between paragraphs without indentation

\usepackage[utf8]{inputenc}    
\usepackage[T1]{fontenc} 

\makeatletter
\def\thm@space@setup{%
	\thm@preskip=\parskip \thm@postskip=0pt
}
\makeatother %make parskip useful for theorems, definitions...

\theoremstyle{plain}
\newtheorem{theorem}{Theorem}[section]
\newtheorem{lemma}{Lemma}[section]

\newtheorem{proposition}{Proposition}[section]

\newtheorem*{claim*}{Claim}
\newtheorem*{lemma*}{Lemma}
\newtheorem*{theorem*}{Theorem}
\newtheorem*{questions*}{Questions}
\theoremstyle{plain}
\newtheorem{thm}{Theorem}

\newtheorem{prop}[thm]{Proposition}
\theoremstyle{definition}
\newtheorem{definition}{Definition}[section]
\newtheorem{example}{Example}[section]
\newtheorem*{defn}{Definition}
\theoremstyle{remark}
\newtheorem{remark}{Remark}

\DeclareMathOperator{\Hessian}{Hess}
\DeclareMathOperator{\distance}{dist}
\DeclareMathOperator{\support}{supp}

\renewcommand{\Im}{\operatorname{Im}}
\renewcommand{\Re}{\operatorname{Re}}

\usepackage[
backend=biber,  %(arxiv)%don't forget to change the setting of texstuido from biber to bibtex
natbib=true,
url=false, 
]{biblatex}
\begin{document}
	
	\title{Geometric Analysis on the Diederich--Forn\ae ss Index}
	
%	\author{Steven G. Krantz\\ \href{mailto: sk@math.wustl.edu}{sk@math.wustl.edu}
%		\and
%		Bingyuan Liu\\ \href{mailto: bingyuan@ucr.edu}{bingyuan@ucr.edu}
%		\and
%		Marco Peloso\\ \href{marco.peloso@unimi.it}{marco.peloso@unimi.it}
%	}
	
		\author{Steven G. Krantz\\ Mathematics Department, Washington University, St. Louis\\ sk@math.wustl.edu
			\and
			Bingyuan Liu\\Mathematics Department, University of California, Riverside\\ bingyuan@ucr.edu
			\and
			Marco Peloso\\Mathematics Department, Universit\'{a} degli Studi di Milano, Milano\\ marco.peloso@unimi.it
		}

	\date{\today}
	
	\makeatletter
	\newcommand{\subjclass}[2][1991]{%
		\let\@oldtitle\@title%
		\gdef\@title{\@oldtitle\footnotetext{#1 \emph{Mathematics subject classification.} #2}}%
	}
	\newcommand{\keywords}[1]{%
		\let\@@oldtitle\@title%
		\gdef\@title{\@@oldtitle\footnotetext{\emph{Key words and phrases.} #1.}}%
	}
	\makeatother	
	
	
	\maketitle
	
	\begin{abstract}
		Given bounded pseudoconvex domains in 2-dimensional complex Euclidean space. We derive analytical and geometric conditions which guarantee the Diederich-Forn\ae ss index is 1. The analytical condition is independent of strongly pseudoconvex points and extends Forn\ae ss--Herbig's theorem in 2007. The geometric condition reveals the index reflects topological properties of boundary. The proof uses an idea including differential equations and geometric analysis to find the optimal defining function. We also give a precise domain of which the Diederich--Forn\ae ss index is 1. The index of this domain can not be verified by formerly known theorems. 
	\end{abstract}

\section{Introduction}\label{intro}
Let $\Omega$ be a bounded pseudoconvex domain in $\mathbb{C}^n$ with
smooth boundary.  
It is well known that such domain $\Omega$ admits a
plurisubharmonic function $-\log(-\delta(z))$, 
where $\delta$ is the signed distance function, that is, 
\[\delta(z):=\begin{cases}
-\distance(z, \partial\Omega) & z\in\Omega\\
\distance(z, \partial\Omega) & \text{otherwise}.
\end{cases}\]
However, the function $-\log(-\delta(z))$ is unbounded when $z$ approaches the boundary,
which make some analysis on the boundary of the domain intractable. 
In 1977, Diederich and Forn\ae ss showed in \citep{DF77b}, that on  
any bounded pseudoconvex domain with smooth boundary there exists
a bounded, plurisubharmonic 
exhaustion function. 
Their idea was to replace
 $-\log(-\delta(z))$ with
$-(-\rho)^\eta$, where $\rho$ is 
some defining function for $\Omega$ and $0<\eta<1$.  In fact, they proved
that,
on any smoothly bounded pseudoconvex domain $\Omega$, with definining
function $\rho$, there exists $0<\eta\le1$ such that $-(-\rho)^\eta$
is a strictly plurisubharmonic exhaustion function. 
Observe that $-(-\rho)^\eta$ will approach 0 when $z$ goes to boundary, even if it will 
not be smooth at the boundary. 

The existence of bounded plurisubharmonic exhaustion functions 
was later generalized to $C^1$ boundary by Kerzman and Rosay in \citep{KR81} and to 
Lipschitz boundary by Demailly in \citep{De87} (see Harrington
\citep{Ha08} too). 
Recently, Harrington generalized the existence theorem to 
$\mathbb{CP}^n$ in \citep{Ha15}. For discussions in $\mathbb{CP}^n$, the reader is also referred to \citep{OS98} and \citep{OS00} by  Ohsawa and Sibony.

%%%%%%LATER

%% Moreover, the Diederich-Forn\ae ss index is known to be related to
%% $\overline{\partial}$-Neumann problem closely. For
%% $\overline{\partial}$-Neumann problem, see the book written by
%% Straube \citep{St10} and the survey of Boas and Straube
%% \citep{BS99}.  

In this paper, we study properties 
of  a given domain $\Omega$, in connection with
the optimization of the exponent in $-(-\rho)^\eta$. 
We now introduce the Diederich-Forn\ae ss index.

\begin{defn}\label{df}
Let $\Omega$ be a bounded, pseudoconvex domain in $\mathbb{C}^n$. The number 
$0 < \tau_\rho < 1$ is called a {\it Diederich-Forn\ae ss exponent} if there exists a defining 
function $\rho$ of $\Omega$ so that $-(-\rho)^{\tau_\rho}$ is plurisubharmonic. 
The index 
$$
\eta:=\sup \tau_\rho \, ,
$$
where the supremum is taken over all defining functions of $\Omega$, is 
called the {\it Diederich-Forn\ae ss index} of the domain $\Omega$.
\end{defn}

As an indication of the importance of the Diederich-Forn\ae ss index
of $\Omega$ we mentioned that 
Berndtsson and Charpentier \cite{BC00}
and Kohn \cite{Ko99}, with two completely different methods,
showed that,
if  $\Omega$ is smooth, bounded and pseudoconvex, then there exists 
$0<s_\Omega\le
+\infty$  such that the Bergman
projetion $P:W^s(\Omega)\to W^s(\Omega)$ is bounded if $0<s<s_\Omega$,
where $W^s(\Omega)$ denotes the classical Sobolev space. 
Berndtsson and Charpentier 
showed that $s_\Omega\ge\eta/2$, where
$\eta$ is the Diederich-Forn\ae ss index of $\Omega$.  
On the other hand, Kohn provided an estimated for $s_\Omega$ again in
terms of the Diederich-Forn\ae ss index of $\Omega$, although in a
less explicit fashion; see also the paper \cite{PZ14}.  

In an earlier paper, Boas and Straube proved that if 
$\Omega$ is a smooth, bounded, pseudoconvex domain in $\mathbb{C}^n$
admitting a defining function that is plurisubharmonic, 
then the Bergman projection $P:C^\infty(\overline{\Omega})\to
C^\infty(\overline{\Omega})$ is bounded, that is, $\Omega$ {\em
  satisfies conditions R}.  Clearly, for such domains, the 
Diederich-Forn\ae ss index is 1.

In \citep{FH07} and \cite{FH08} Forn\ae ss--Herbig addressed the question
whether a a smooth, bounded, pseudoconvex domain in $\mathbb{C}^2$ and 
$\mathbb{C}^n$, respectively, possessing a defining function that is
plurisubharmonic on the boundary has Diederich-Forn\ae ss index equal to 1.
They answered this question in the positive.
 The converse does not hold in general. That is, if a domain has
 Diederich-Forn\ae ss index 1, it does not necessarily admit a
 defining function which is plurisubharmonic on the boundary. The
 latter statement was proved by Behrens in \citep{Be85} where she gave an
 example of a bounded domain
with real analytic boundary and 
not having any local defining function that is plurisubharmonic on
 near a fixed boundary point.   Nonetheless, this domain
 has
 Diederich-Forn\ae ss index 1.
The conclusion follows from another, related work  by Diederich and
Fornaess \citep{DF77c}, where
they showed the Diederich-Forn\ae ss index is 1 if the pseudoconvex
domain is {\em regular}, see Definition 1 and Theorem 1 in
\citep{DF77c}.

The main goal of this paper  to extend 
Forn\ae ss--Herbig's result. More precisely, we would like to address the
following questions: 

\begin{questions*}
	\begin{enumerate}
		\item\label{1} Can one find a more general condition than
                  plurisubharmonicity of a defining function on the boundary
                  to guarantee the Diederich--Forn\ae ss index is $1$?
                  Possibly, this condition should cover the example of Behrens.  
		\item On the other hand, how can one realize the condition from a geometric point of view? 
		\item Can one find a
                  bounded pseudoconvex domain admitting
                  Diederich--Forn\ae ss index 1, of which the fact is
                  not discovered by formerly known theorems. In other words, we want to see a new application of the condition we found in Question \ref{1} and this application should be new to us.
	\end{enumerate}
\end{questions*}

The Question 1 is necessary to the Diederich--Forn\ae ss index,
because the condition of Forn\ae ss--Herbig is not sharp. We need
to find a sufficient condition cover the example of Behrens at
least. Indeed, the following theorem is an extension of Forn\ae
ss--Herbig's theorem. The proof will be in Section
\ref{sec3}. Please also have a look at Section \ref{pre} and Section
\ref{SC} for basic notations. 

\begin{thm}\label{prop}
	Let $\Omega$ be a bounded domain with smooth boundary in $\mathbb{C}^2$. Let $\Sigma$ denote the Levi-flat set in $\partial\Omega$. Assume that there exists a defining function $\rho$ of $\Omega$ such that, on $\Sigma$, we have the condition $\Hessian_r(L, N)=0$ where $L$ is the normalized holomorphic tangential vector field of $\partial\Omega$ and $N$ is the normalized complex normal vector field of $\partial\Omega$. Then the Diederich-Forn\ae ss index of $\Omega$ is $1$.
\end{thm}

\begin{remark}
	In practice, we do not need to assume that the $L, N$ are normalized vectors. This is because $\Hessian_\rho (L,N)$ is tensorial, that is, $\Hessian_\rho (fL,gN)=f\bar{g}\Hessian_\rho (L,N)=0$ for arbitrary functions $f$ and $g$.
\end{remark}

The preceding theorem not only extends Forn\ae ss--Herbig's theorem, but also relates more geometric informations to the index. This connects the Deiderich--Forn\ae ss to Question 2. For this aim, we have to introduce some of our conventions. Namely, we will call a simple curve a \textit{real curve} if it can be parametrized by a smooth map $\Psi: t\mapsto\mathbb{C}^2$. Also, for the definition and discussion of transversality, see Section \ref{pre}. 

We are ready to answer Question 2 with a series of results as what follows. All of these will be discussed in Section \ref{sec4}.
\begin{thm}\label{main}
	Let $\Omega$ be a bounded domain with smooth boundary in $\mathbb{C}^2$. 
	Let $\Sigma$ denote the set of Levi-flat points in $\partial\Omega$. 
	Assume that $\Sigma$ is a real curve and transversal to the holomorphic tangent vector of $\partial\Omega$. 
	Then the Diederich-Forn\ae ss index of $\Omega$ is $1$.
\end{thm}

\begin{remark}\label{rmk}
In particular, as a consequence we obtain that if
 $\Omega$ be a bounded domain with smooth boundary in $\mathbb{C}^2$,
and  the set of Levi-flat points in $\partial\Omega$
$\Sigma$ is a set of isolated points, then
 the Diederich-Forn\ae ss index of $\Omega$ is $1$.

	In fact, it is not hard to see also that if the set of Levi-flats points consists of finite many isolated points and
	finite many disjoint real curves transversal to the holomorphic vector fields, then the Diederich-Forn\ae ss index is 1.
\end{remark}

%\begin{remark}
%	It can also seen that the same argument is valid on a K\"{a}hler manifold with some curvature constraint. For example, let $\Omega$ be a bounded domain with smooth boundary in a 2-dimensional complex manifold $M$ and assume on a neighborhood of $\Omega$ in $M$, a K\"{a}hler metric with vanishing sectional curvature exists.
%	Let $\Sigma$ denote the set of Levi-flat points in $\partial\Omega$. 
%	Assume that $\Sigma$ is a real curve and transversal to the holomorphic tangent vector of $\partial\Omega$. 
%	Then the Diederich-Forn\ae ss index of $\Omega$ is $1$. The condition of K\"{a}hler metric with vanishing sectional curvature might be replaced by the one of K\"{a}hler metric (see Section \ref{SC}) satisfying \[\nabla_L\nabla_N\nabla\rho-\nabla_N\nabla_L\nabla\rho-\nabla_{[L,N]}\nabla\rho=0.\]For the definition of Diederich-Forn\ae ss index on a K\"{a}hler manifold, see \citep{Bi14} and its erratum on her homepage. 
%\end{remark}

Moreover, Theorem \ref{main} is a special case of the following proposition. Indeed, Proposition \ref{main2} describes geometry of the Levi-flat sets by existence of solution to a type of partial differential equations.

\begin{prop}\label{main2}
	Let $\delta$ be an arbitrarily defining function of a bounded domain $\Omega\subset\mathbb{C}^2$ with smooth boundary. Let $\Sigma\subset\partial\Omega$ denote the Levi-flat sets of $\partial\Omega$. Suppose there is a real function $u$ which solves \[L(u)=-\frac{\Hessian_\delta(L, N)}{\|\nabla\delta\|}\] on $\Sigma$. Then the Diederich-Forn\ae ss index of $\Omega$ is $1$.
\end{prop}

In Section \ref{sec5}, we construct a specific bounded pseudoconvex domain $\tilde{\Omega}$ to answer Question 3. We remind the reader that our example cannot be verified by all known theorems except ours. Finally, Theorem \ref{example} gives a satisfaction answer.

Before we proceed to prove our theorems, we briefly mention some history here and from it, one can have a full picture of the other extreme cases in which the Diederich-Forn\ae ss index is away from $1$. In 1977, Diederich-Forn\ae ss found a domain called the worm domain in \citep{DF77a} 
which gives a non-trivial Diederich-Forn\ae ss index (i.e., an index strictly between
0 and 1). In fact, they show that the Diederich-Forn\ae ss exponent
can be arbitrarily close to $0$, see \citep{DF77a}. 

In 1992, Barrett showed in \citep{Ba92}, that the Bergman projection $P$ on $\Omega_\beta$ does not map the Sobolev space $W^k(\Omega_\beta)$ into $W^k(\Omega_\beta)$ when $k\geq\pi/(2\beta-\pi)$. In 2000, Berndtsson and Charpentier showed, in \citep{BC00}, that the Bergman projection $P$ on $\Omega_\beta$ does map Sobolev space $W^k(\Omega_\beta)$ into $W^k(\Omega_\beta)$ when $k< \tau/2$ where $\tau$ is a Diederich-Forn\ae ss exponent. As a consequence, the Diederich-Forn\ae ss index of $\Omega_\beta$ is less or equal to
$2\pi/(2\beta-\pi)$. The reader can also deduce this result from
Krantz and Peloso \citep{KP08}. Indeed, Theorem 6 in \citep{DF77a}
says that if the standard defining function of $\Omega_\beta$ has
exponent $\leq\eta$, then all other defining functions have exponent
$\leq\eta$, that is, the Diederich-Forn\ae ss index of
$\Omega_\beta\leq\eta$. Thus, the calculation in \citep{KP08} shows
that the $\Omega_\beta\leq\pi/(2\beta-\pi)$. Recently Fu and Shaw and
Adachi and Brinkschulte proved independently in \citep{FS14} and
\citep{AB14} respectively that, roughly speaking, if a relatively
compact domain in a complex manifold has all boundary points
Levi-flat, then the Diederich-Forn\ae ss index is non-trivial. Also,
two papers of Herbig--McNeal in
\citep{HM12a} and \citep{HM12b} include some interesting results.

\section{Preliminaries}\label{pre}
We begin by fixing some basic notation. Let $M$ be a Hermitian manifold with complex structure $J$ and metric $g$. For a real tangent vector fields $X$ we define  \[Z=\frac{1}{2}(X-\sqrt{-1}JX)\] to be a holomorphic tangent vector field and \[\overline{Z}=\frac{1}{2}(X+\sqrt{-1}JX)\] to be an anti-holomorphic tangent vector field. Recall that, if $f$ is a function defined on $M$, then  \[Zf=g\left(\nabla f, \bar{Z}\right)=g\left(Z, \nabla f\right).\]

We also define the Hessian of a function $f$ on real tangent vector fields:
\[\Hessian_f(X, Y)=g(\nabla_X\nabla f, Y)=Y(Xf)-(\nabla_YX)f,\] and for holomorphic tangent vectors we calculate as follows:
\[\Hessian_f(Z, W)=g(\nabla_Z\nabla f, W)=Z(\overline{W}f)-\nabla_Z\overline{W}f=\overline{\Hessian_f(W,Z)}.\]
We can also write the gradient in complex notation. Namely, \[\nabla f=2\left(\frac{\partial f}{\partial z}\frac{\partial}{\partial\bar{z}}+\frac{\partial f}{\partial \bar{z}}\frac{\partial}{\partial z}+\frac{\partial f}{\partial w}\frac{\partial}{\partial\bar{w}}+\frac{\partial f}{\partial \bar{w}}\frac{\partial}{\partial w}\right).\]

If the sectional curvature of $M$ vanishes, then we have that the curvature tensor vanishes which means
\[R_m(Z_1, Z_2, Z_3, Z_4)\equiv0,\] where $R_m$ denotes the curvature tensor.
That means \[0\equiv \nabla_{Z_1}\nabla_{Z_2} Z_3-\nabla_{Z_2}\nabla_{Z_1} Z_3-\nabla_{[Z_1,Z_2]} Z_3,\] for arbitrary holomorphic tangent fields $Z_1, Z_2, Z_3$ of $M$. For the basic notion of curvatures see \citep{Pe06}.

From now on, we work on a domain in $\mathbb{C}^2$ and discuss transversality. Recall that a tangent vector of $\mathbb{C}^2$
\[L=f_1(z, w)\frac{\partial}{\partial z}+f_2(z, w)\frac{\partial}{\partial \bar{z}}+g_1(z,w)\frac{\partial}{\partial w}+g_2(z,w)\frac{\partial}{\partial \bar{w}}\]
indeed defines two real tangent vectors:
\[
2\Re L:=\Re (f_1+f_2)\frac{\partial}{\partial x}+\Im(f_1-f_2)\frac{\partial}{\partial y}+\Re(g_1+g_2)\frac{\partial}{\partial u}+\Im(g_1-g_2)\frac{\partial}{\partial v}
\]
and 
\[
2\Im L:=\Im (f_1+f_2)\frac{\partial}{\partial x}+\Re(f_2-f_1)\frac{\partial}{\partial y}+\Im(g_1+g_2)\frac{\partial}{\partial u}+\Re(g_2-g_1)\frac{\partial}{\partial v},
\]

where we let the coordinate be \[z=x+yi, \text{   and   }w=u+vi.\]

Sometimes $\Re L$ and $\Im L$ are linearly dependent. But, for a nonzero holomorphic tangent vector \[L=f(z, w)\frac{\partial}{\partial z}+g(z,w)\frac{\partial}{\partial w},\]
$\Re L$ and $\Im L$ are always independent because of the following easy lemma.

\begin{lemma}\label{wellde}
Let \[V=f_1(z, w)\frac{\partial}{\partial z}+f_2(z, w)\frac{\partial}{\partial \bar{z}}+g_1(z,w)\frac{\partial}{\partial w}+g_2(z,w)\frac{\partial}{\partial \bar{w}},\] be a complex vector field. If $\Re V$ and $\Im V$ are linearly dependent, then \[|f_1|=|f_2|\quad\text{ and }\quad |g_1|=|g_2|.\] In particular, if $V=L$ is a holomorphic vector field with $\Re L$ and $\Im L$ linearly dependent, then $L=0$.
\end{lemma}
\begin{proof}
	Give that \[V=f_1(z, w)\frac{\partial}{\partial z}+f_2(z, w)\frac{\partial}{\partial \bar{z}}+g_1(z,w)\frac{\partial}{\partial w}+g_2(z,w)\frac{\partial}{\partial \bar{w}},\]
	and that \[\begin{pmatrix}
	\Re (f_1+f_2)\\\Im(f_1-f_2)\\\Re(g_1+g_2)\\\Im(g_1-g_2)
	\end{pmatrix}\quad\text{and}\quad\begin{pmatrix}
	\Im (f_1+f_2)\\\Re(f_2-f_1)\\\Im(g_1+g_2)\\\Re(g_2-g_1)
	\end{pmatrix}\]
	are  linearly dependent. Hence, both of the following two determinants 
\[\begin{vmatrix}
	\Re (f_1+f_2)&  \Im (f_1+f_2)\\ \Im(f_1-f_2) & \Re(f_2-f_1)
\end{vmatrix}\quad\text{or}\quad\begin{vmatrix}
\Re(g_1+g_2)&\Im(g_1+g_2)\\ \Im(g_1-g_2) &\Re(g_2-g_1)
\end{vmatrix}\]
have to be zero.

By straightforward calculation, \[\begin{vmatrix}
\Re (f_1+f_2)&  \Im (f_1+f_2)\\ \Im(f_1-f_2) & \Re(f_2-f_1)
\end{vmatrix}=|\Re f_2|^2-|\Re f_1|^2-|\Im f_1|^2+|\Im f_2|^2=|f_2|^2-|f_1|^2\]
and
\[\begin{vmatrix}
\Re(g_1+g_2)&\Im(g_1+g_2)\\ \Im(g_1-g_2) &\Re(g_2-g_1)
\end{vmatrix}=|\Re g_2|^2-|\Re g_1|^2-|\Im g_1|^2+|\Im g_2|^2=|g_2|^2-|g_1|^2.\]
This immediately gives \[|f_1|=|f_2|\quad\text{ and }\quad |g_1|=|g_2|.\]
\end{proof}

Lemma \ref{wellde} guarantees that we are able to generalize the notion of \textit{transverality} to nonzero holomorphic tangent vector fields, because $|f_2|=|g_2|=0$.

\begin{definition}
	We say that a real curve $\gamma(t)$ is \textit{transversal} to a nonzero holomorphic tangent vector \[L=f(z, w)\frac{\partial}{\partial z}+g(z,w)\frac{\partial}{\partial w}\] 
	if $\gamma'(t)$, $\Re L$ and $\Im L$ are linear independent.
\end{definition}
It is also easy to see that linear independence is preserved by a diffeomorphism.

\section{Calculation of the D-F Index} \label{SC}\label{sec3}
Let $r$ be an arbitrary defining function of $\Omega$. We want to modify the defining function in order to seek the best one for optimizing the Diederich-Forn\ae ss exponent. Put $\rho=re^{\psi}$, where $\psi$ will be determined later. 

We first introduce some definitions.
	\begin{definition}
		Let $\Omega$ be a bounded domain with smooth boundary in $\mathbb{C}^2$ defined by a smooth defining function $\rho$. The vector field \[L=\frac{1}{\sqrt{|\frac{\partial \rho}{\partial z}|^2+|\frac{\partial \rho}{\partial w}|^2}}(\frac{\partial \rho}{\partial w}\frac{\partial}{\partial z}-\frac{\partial \rho}{\partial z}\frac{\partial}{\partial w})\] on $\partial\Omega$ is called the \textit{normalized holomorphic tangential vector field }, and \[N=\frac{1}{\sqrt{|\frac{\partial \rho}{\partial z}|^2+|\frac{\partial \rho}{\partial w}|^2}}(\frac{\partial \rho}{\partial\bar{z}}\frac{\partial}{\partial z}+\frac{\partial \rho}{\partial\bar{w}}\frac{\partial}{\partial w})\] on $\partial\Omega$ is called the \textit{normalized complex normal vector field}.
	\end{definition}

	Note that, to the fact that $L, N$ are unit vectors, $\Hessian_{|z|^2}(L, L)=\Hessian_{|z|^2}(N, N)=1$. Also $\Hessian_{|z|^2}(L, N)=0$ due to the fact they are orthogonal. Here $|z|^2$ should be read as $|(z, w)|^2$, but for concision, we will not write it as $|(z, w)|^2$. 

The following lemma is proved by a direct calculation. Since the calculation is tedious, we put it in appendix. 

\begin{lemma}\label{marco}
		Let $\Omega, r, L$ and $N$ be as above, and let $\psi$ be a smooth function. Let
		$\eta, \delta>0$.  Then
		\begin{align*}
		& \Hessian_{-(-re^\psi)^\eta e^{-\delta\eta |z|^2}}(aL+bN,aL+bN) \\
		&\qquad = -\eta e^{-\delta\eta |z|^2}(-re^\psi)^{\eta-1}
		\left(|a|^2I+2\Re (a\bar{b}I\!I)+|b|^2I\!I\!I\right) \,,
		\end{align*}
		where, on a sufficiently small neighborhood of $\partial\Omega$ in $\mathbb{C}^2$, 
		\[\begin{split}
		I &= e^\psi\Big((\delta^2\eta)(-r) L(|z|^2)\overline{L} (|z|^2)-
		\delta(-r)+r L(\psi)\overline{L}(\psi)-\Hessian_r(L,L)-r \Hessian_\psi(L, L)\Big),\\
		I\!I\!I
		&< \frac{e^\psi}{2(-r)}(\eta-1)|N(r)|^2
		\end{split}\]
		and on $\partial\Omega$, we have the estimate
		\[|I\!I| <e^\psi\Big(\delta\eta |L(|z|^2)\overline{N}
		(r)|+|L(\psi)\overline{N}(r)|+|\Hessian_r(L, N)|\Big) .
		\]\qed
\end{lemma}

We are ready to define $\psi=-C|\Hessian_r(L_r, N_r)|^2$, where $C>0$ is some number to be determined and  \[L_r=\frac{1}{\sqrt{|\frac{\partial r}{\partial z}|^2+|\frac{\partial r}{\partial w}|^2}}(\frac{\partial r}{\partial w}\frac{\partial}{\partial z}-\frac{\partial r}{\partial z}\frac{\partial}{\partial w})\] and \[N_r=\frac{1}{\sqrt{|\frac{\partial r}{\partial z}|^2+|\frac{\partial r}{\partial w}|^2}}(\frac{\partial r}{\partial\bar{z}}\frac{\partial}{\partial z}+\frac{\partial r}{\partial\bar{w}}\frac{\partial}{\partial w}).\] This definition of $\psi$ is originally due to Forn\ae ss--Herbig in \citep{FH07}. More specifically, they proved the following lemma in \citep{FH07}. Here we rewrite it with a language of differential geometry. For the detail of the proof, please see the Appendix.

\begin{lemma}\label{3.1}
	Assume that \[\Hessian_r(L, N)=0,\] on $\Sigma$, where $\Sigma$ is a subset of $\partial\Omega$.
	Let \[\psi=-C|\Hessian_r(L_r, N_r)|^2\] for arbitrary $C>0$. Then
\begin{equation}\label{eqn0} 
		L_r(\psi)=0,
		\end{equation} 
		and
		\begin{equation}\label{eqnl}
		\Hessian_\psi(L, L)=\Hessian_\psi(L_r,L_r)\leq -C|L_r\Hessian_r(N_r, L_r)|^2=-C|N_r\Hessian_r(L_r, L_r)|^2
		\end{equation} on $\Sigma$.\qed
\end{lemma}

\begin{remark}
	From the preceding lemma, we can also see that on $\Sigma$
\[
			L(\psi)=0,
\]
			and
		\[
			\Hessian_\psi(L, L)\leq -C|L\Hessian_r(N, L)|^2=-C|N\Hessian_r(L, L)|^2
	\]
because on $\partial\Omega$, $L, N$ coincide with $L_r, N_r$.
\end{remark}

Then, with the notation of Lemma \ref{marco}, we have on $\Sigma$,
\[|I\!I|<e^\psi\big(\delta\eta |L(|z|^2)\overline{N} (r)|\big).\]

Moreover, there must be a neighborhood $\Sigma_\epsilon$, which is dependent on $\epsilon>0$, of $\Sigma$ in $\mathbb{C}^2$, and on the neighborhood, we have
\[|I\!I|<e^\psi\big(3\delta\eta  \max\{L(|z|^2), \epsilon\}|\overline{N} (r)|\big)\]
for some $\epsilon>0$.

We can now prove Theorem \ref{prop}. We want to point out that Behren's counterexample that we mentioned in Section \ref{intro} will not contradict the converse of Theorem \ref{prop}. This is because her example has only one Levi-flat point on the boundary and we will show more generally that, if the Levi-flat points form a real curve (see Theorem \ref{main}), then it satisfies the condition of Theorem \ref{prop}. 

We prove the theorem by modifying the argument of Fornaess--Herbig in \citep{FH07}. Our proof has a few new arguments. We need extra estimates on the points with Levi-forms bounded below. We also need to consider the points which have small positive Levi-forms. 

\begin{proof}[Proof of Theorem \ref{prop}]

Let $\psi$ be defined in Lemma \ref{3.1}. Firstly, we claim that if the Levi-form is bounded below by a positive number $\alpha>0$, then in a neighborhood of these boundary points in $\mathbb{C}^2$ \[\Hessian_{-(-\rho)^\eta}(aL+bN,aL+bN)>0\] holds for any defining function $\rho$ of $\Omega$ and any $0<\eta<1$. We are going to show this fact in the following paragraph.

It is enough to show that the complex Hessian of $-(-\rho)^\eta$ is positive definite in a neighborhood of the strongly pseudoconvex boundary points. Rewrite the complex Hessian of $-(-\rho)^\eta$ with matrices.
\[\begin{split}
\Hessian_{-(-\rho)^\eta}&=\begin{pmatrix}
\Hessian_{-(-\rho)^\eta}(L,L)&\Hessian_{-(-\rho)^\eta}(L,N)\\
\Hessian_{-(-\rho)^\eta}(N,L)&\Hessian_{-(-\rho)^\eta}(N,N)
\end{pmatrix}\\&=\begin{pmatrix}
\eta(-\rho)^{\eta-1}\Hessian_\rho(L, L)&\eta(-\rho)^{\eta-1}\Hessian_\rho(L, N)\\
\eta(-\rho)^{\eta-1}\Hessian_\rho(N, L)&\eta(-\rho)^{\eta-1}\left(\Hessian_\rho(N, N)+\frac{1-\eta}{-\rho}|N\rho|^2\right)\end{pmatrix}\\
&=\eta(-\rho)^{\eta-1}\begin{pmatrix}
\Hessian_\rho(L, L)&\Hessian_\rho(L, N)\\
\Hessian_\rho(N, L)&\Hessian_\rho(N, N)+\frac{1-\eta}{-\rho}|N\rho|^2
\end{pmatrix}
\end{split}\]
is positive definite if and only if $\Hessian_\rho(L, L)>0$ and 
\[\begin{split}
&\begin{vmatrix}
\Hessian_\rho(L, L)&\Hessian_\rho(L, N)\\
\Hessian_\rho(N, L)&\Hessian_\rho(N, N)+\frac{1-\eta}{-\rho}|N\rho|^2
\end{vmatrix}\\=&\Hessian_\rho(L, L)\left(\Hessian_\rho(N, N)+\frac{1-\eta}{-\rho}|N\rho|^2\right)-|\Hessian_\rho(L, N)|^2\\=&\Hessian_\rho(L, L)\Hessian_\rho(N, N)+\Hessian_\rho(L, L)\frac{1-\eta}{-\rho}|N\rho|^2-|\Hessian_\rho(L, N)|^2>0
\end{split}.\]
But this is clear because of $\Hessian_\rho(L, L)\geq\frac{\alpha}{2}$ and \[\Hessian_\rho(L, L)\frac{1-\eta}{-\rho}|N\rho|^2\geq \frac{\alpha}{2}\frac{1-\eta}{-\rho}|N\rho|^2\] in a neighborhood of strongly pseudoconvex boundary points. Indeed, the term $\frac{1-\eta}{-\rho}|N\rho|^2$ can approach $+\infty$ as $\rho$ goes to $0$, while \[\Hessian_\rho(L, L)\Hessian_\rho(N, N)-|\Hessian_\rho(L, N)|^2\] has to be bounded which completes the proof of the fact.

This implies that
\[\Hessian_{-(-re^\psi)^\eta e^{-\delta\eta |z|^2}}(aL+bN,aL+bN)>0.\]
So we just need to prove that \[\Hessian_{-(-re^\psi)^\eta e^{-\delta\eta |z|^2}}(aL+bN,aL+bN)>0\] on a neighborhood of the Levi-flat points in $\mathbb{C}^2$.

Let $\Sigma$ be the Levi-flat subsets of $\partial\Omega$. We learned in Lemma \ref{marco} that there exists a neighborhood of $\partial\Omega$ in $\mathbb{C}^2$ such that, on this neighborhood, \[I\!I\!I<\frac{e^\psi}{-2r}(\eta-1)|N(r)|^2.\] 

To prove
\[\Hessian_{-(-re^\psi)^\eta e^{-\delta\eta |z|^2}}(aL+bN,aL+bN)>0\] on a neighborhood of $\Sigma$ for an appropriate $\delta$, for all $a,b\in\mathbb{C}$ and all $0<\eta<1$, we need to show that \[|a|^2I+|b|^2I\!I\!I<-2|ab||I\!I|\] on a neighborhood of $\Sigma$ in $\mathbb{C}^2$. We have seen on the neighborhood $\Sigma_\epsilon$ of $\Sigma$ in $\mathbb{C}^2$, \[I\!I\!I<\frac{e^\psi}{-2r}(\eta-1)|N(r)|^2\] and \[|I\!I|<e^\psi\big(3\delta\eta  \max\{L(|z|^2), \epsilon\}|\overline{N} (r)|.\] If we can show that, on a neighborhood of $\Sigma$ in $\mathbb{C}^2$,  \[I<-\frac{e^\psi}{4}\delta(-r),\] then we are able to see that \[\Hessian_{-(-re^\psi)^\eta e^{-\delta\eta |z|^2}}(aL+bN,aL+bN)>0.\]
This is because, after shrinking $\delta$ further,
\begin{equation}\label{1/4}
-|a|^2\frac{1}{4}\delta(-r)-|b|^2\frac{1-\eta}{-2r}|N(r)|^2<-|ab||N(r)|\frac{\sqrt{1-\eta}}{2}\sqrt{\delta}<-6|ab|\delta\eta \max\{L(|z|^2), \epsilon\}|\overline{N} (r)|
\end{equation}
for some $\epsilon>0$.

Thus, if we can show that, on a neighborhood of $\Sigma$,
\[
(\delta^2\eta)(-r) L(|z|^2)\overline{L} (|z|^2)-\delta(-r)+r L(\psi)\overline{L}(\psi)-\Hessian_r(L,L)-r \Hessian_\psi(L, L)<-\frac{1}{4}\delta(-r),
\] 
then we are done.

Indeed, after shrinking $\delta$, we do not need to consider $(\delta^2\eta)(-r) L(|z|^2)\overline{L} (|z|^2)$ because it is $o(\delta^2)$ and we can control it to make \[(\delta^2\eta)(-r) L(|z|^2)\overline{L} (|z|^2)=(\delta^2\eta)(-r) |L(|z|^2)|^2<\frac{1}{2}\delta(-r).\]

Hence we just need to show that
\[r(q)L(\psi)\vert_q\overline{L}(\psi)\vert_q-\Hessian_r(L,L)\vert_q-r(q) \Hessian_\psi(L, L)\vert_q<\frac{1}{4}\delta(-r(q))\Hessian_{|z|^2}(L, L)\vert_q\]
for $q\in\Omega$. And thus it is enough to show that
\[r(q)L(\psi)\vert_q\overline{L}(\psi)\vert_q-\Hessian_r(L,L)\vert_q+\Hessian_r(L,L)\vert_p-r(q) \Hessian_\psi(L, L)\vert_q<\frac{1}{4}\delta(-r(q))\]
for $q$ in some neighborhood of $\partial\Omega$ in $\mathbb{C}^2$ and $p\in\partial\Omega$ so that $\distance(q,p)=\distance(q,\partial\Omega)$, because of $\Hessian_r(L,L)\vert_p\geq 0$.

Thus it is enough to show that
\[-L(\psi)\vert_q\overline{L}(\psi)\vert_q-\lim\limits_{q\to p}\frac{\Hessian_r(L,L)\vert_q-\Hessian_r(L,L)\vert_p}{-r(q)+r(p)}+\Hessian_\psi(L, L)\vert_p<\frac{1}{5}\delta\]
on a neighborhood of $\Sigma$ in $\partial\Omega$ where the limit means that $q$ approaches $p$ along the shortest path.

But here

\[\begin{split}
-\lim\limits_{q\to p}\frac{\Hessian_r(L,L)\vert_q-\Hessian_r(L,L)\vert_p}{-r(q)+r(p)}=&\frac{g(\nabla \Hessian_r(L,L), N+\overline{N})}{g(\nabla r, N+\overline{N})}=\frac{2\Re(N\Hessian_r(L,L))}{\|\nabla r\|}\\\leq &\mathcal{K}\Re(N\Hessian_r(L,L))
\end{split}\]
for some $\mathcal{K}>0$.

It is enough to show that
\[\mathcal{K}|N\Hessian_r(L,L)|+\Hessian_\psi(L, L)<\frac{1}{8}\delta\]
on $\Sigma$ because  \[L(\psi)\vert_q\overline{L}(\psi)\vert_q=|L(\psi)\vert_q|^2\] is nonnegative.

We calculate
\[\begin{split}
\mathcal{K}|N\Hessian_r(L,L)|+\Hessian_\psi(L, L)\leq&\mathcal{K}|N\Hessian_r(L,L)|-C|N\Hessian_r(L,L)|^2\\<&\frac{\mathcal{K}^2}{4C},
\end{split}\]
and we take $C$ so that \[\frac{\mathcal{K}^2}{4C}\leq \frac{1}{8}\delta.\] Then we have 
\[\mathcal{K}|N\Hessian_r(L,L)|+\Hessian_\psi(L, L)<\frac{1}{8}\delta,\] which completes the proof.
\end{proof}

\section{Complex Transport Equations}\label{sec4}
In this section, we are going to imitate transport equations in the real sense to the complex sense. It is well known that differential equations are very different in the real and complex contents. Thus, the imitation of transport equation in the complex sense cannot be fully extended.
 
In this section, our aim is to show that, if $L$ is a holomorphic (anti-holomorphic) tangent vector field and $h$ is a smooth complex-valued function, the equation 
\[Lu=h\] is always solvable for real $u$. We prove the following lemma.

\begin{proposition}\label{cprop}
	Let $\epsilon>0$ be arbitrary and $\Gamma$ be a compact (real) smooth curve parametrized by $\gamma: [0,r]\rightarrow \Gamma$. (Here $\gamma(0)=\gamma(r)$ or $\gamma(0)\neq\gamma(r)$.) Let $h$ be a complex-valued smooth function defined on a neighborhood of $\Gamma$ in $\partial\Omega$. 
 Then, the following equation 
	\begin{equation}\label{cyn}
		Lu=h, \quad \text{  on  }\Gamma
	\end{equation} 
	has a real solution, where $u$ is real-valued smooth function defined in the neighborhood of $\Gamma$.
\end{proposition}

\begin{proof}
	Assume $\Gamma$ is parametrized by $\Gamma=\gamma ([0,r])$. We can extend $\Gamma$ smoothly beyond its endpoints if it is not a closed curve in order to give a neighborhood of endpoints to study. In case of a non-closed curve, redefine $\gamma: [-\epsilon, r+\epsilon]\rightarrow\partial\Omega$ and still assume that $\gamma$ is a smooth real curve and $\gamma'$, $\Re L$ and $\Im L$ are independent. (This is feasible if $\epsilon>0$ is small enough.) We prove (\ref{cyn}).  Let $p\in\Gamma$. We find a neighborhood $U_p$ of $p$ and assume $\phi_p: U_p\rightarrow \mathbb{R}^3$ is a diffeomorphism. Then we have in the coordinates, the holomorphic tangent vector \[L=f_1(z,t)\frac{\partial}{\partial z}+f_2(z,t)\frac{\partial}{\partial\bar{z}}+g(z,t)\frac{\partial}{\partial t}\]on the neighborhood. This gives two vectors $\Re L$ and $\Im L$. Possibly after shrinking $U_p$, we can always assume that $\phi_p\circ\gamma(t)=(0,0,t)$ and $\phi_p(\Gamma)$ is contained in the $t$-axis of $\mathbb{R}^3$. We first consider the following equation which admits smooth real solution \[(x(s_1,s_2,s_3),y(s_1,s_2,s_3),t(s_1,s_2,s_3)):\]
	\begin{equation}\label{const1}
	\begin{cases}
	\frac{\partial x}{\partial s_1}=\Re(f_1(0,0,s_3)+f_2(0,0,s_3))\\
	\frac{\partial y}{\partial s_1}=\Im(f_1(0,0,s_3)-f_2(0,0,s_3))\\
	\frac{\partial t}{\partial s_1}=\Re g(0,0,s_3),
	\end{cases}
	\end{equation}
	\begin{equation}\label{const2}
	\begin{cases}
	\frac{\partial x}{\partial s_2}=\Im(f_1(0,0,s_3)+f_2(0,0,s_3))\\
	\frac{\partial y}{\partial s_2}=\Re(f_2(0,0,s_3)-f_1(0,0,s_3))\\
	\frac{\partial t}{\partial s_2}=\Im g(0,0,s_3)
	\end{cases}
	\end{equation}
	
	and $(x(s_1,s_2,s_3),x(s_1,s_2,s_3), t(s_1,s_2,s_3))$ also satisfies the initial conditions:
	\[(x(0,0,s_3),y(0,0,s_3), t(0,0,s_3))=\phi_p\circ\gamma(s_3)=(0,0,s_3).\]

	 To solve (\ref{const1}) and (\ref{const2}) in $\phi_p(U_p)\subset\mathbb{R}^3$, we let 
	\[x(s_1, s_2, s_3)=\Re(f_1(0,0,s_3)+f_2(0,0,s_3))s_1+\Im(f_1(0,0,s_3)+f_2(0,0,s_3))s_2,\]
		\[y(s_1, s_2, s_3)=\Im(f_1(0,0,s_3)-f_2(0,0,s_3))s_1+\Re(f_2(0,0,s_3)-f_1(0,0,s_3))s_2\] and \[t(s_1,s_2,s_3)=s_3+(\Re g(0,0,s_3))s_1+(\Im g(0,0,s_3))s_2.\]We check also the initial condition $(x(0,0,s_3),y(0,0,s_3))=(0,0)$ and $t(0,0,s_3)=s_3$. Hence \[(x(0,0,s_3),y(0,0,s_3), t(0,0,s_3))=(0,0,s_3).\]
		Without loss of generality, we consider $x(s_1,s_2,s_3)$ solving
			\begin{equation}
			\begin{cases}
			\frac{\partial x}{\partial s_1}=\Re(f_1(0,0,s_3)+f_2(0,0,s_3))\\
			\frac{\partial x}{\partial s_2}=\Im(f_1(0,0,s_3)+f_2(0,0,s_3))\\
			\end{cases}
			\end{equation}
with initial condition $x(0,0,s_3)=0$. It is clear that the solution is unique. Hence it can be extended to the whole curve $\gamma([0,r])$ uniquely.
	
	We are going to check the condition of the inverse function theorem on the map \[\left(x(s_1,s_2,s_3), y(s_1,s_2,s_3), t(s_1,s_2,s_3)\right)\]
	at $(0,0, s_3)$. Now
	\[\begin{pmatrix}
	\frac{\partial x}{\partial s_1}&\frac{\partial x}{\partial s_2}&\frac{\partial x}{\partial s_3}\\	\frac{\partial y}{\partial s_1}&\frac{\partial y}{\partial s_2}&\frac{\partial y}{\partial s_3}\\	\frac{\partial t}{\partial s_1}&\frac{\partial t}{\partial s_2}&\frac{\partial t}{\partial s_3}
	\end{pmatrix}\]
	has rank $3$ because $\gamma'(s_3)$, $\Re L$ and $\Im L$ are linearly independent. Hence locally we can define \[h(s_1, s_2,s_3):=h(x(s_1, s_2, s_3), y(s_1,s_2,s_3), t(s_1,s_2,s_3)).\]
	Thus, we define
	\[u(s_1,s_2, s_3)=h_1(0,0,s_3)s_1+h_2(0,0,s_3)s_2.\]
We immediately find that \[\frac{\partial u}{\partial s_1}=h_1(0,0,s_3)\]
and\[\frac{\partial u}{\partial s_2}=h_2(0,0,s_3)\] can be solved uniquely by \[u=h_1(0,0,s_3)s_1+h_2(0,0,s_3)s_2\] given with the initial condition \[u(0,0,s_3)=0.\]

	Since \[\Re L=\Re(f_1+f_2)\frac{\partial}{\partial x}+\Im(f_1-f_2)
	\frac{\partial}{\partial y}+\Re g\frac{\partial}{\partial t},\] on $\Gamma$ locally it can be rewritten as 
	\[\begin{split}
	\Re L=&\Re(f_1(0,0,s_3)+f_2(0,0,s_3))\frac{\partial}{\partial x}+\Im(f_1(0,0,s_3)-f_2(0,0,s_3))
	\frac{\partial}{\partial y}+\Re g(0,0,s_3)\frac{\partial}{\partial t}\\=&\frac{\partial x}{\partial s_1}\frac{\partial}{\partial x}+\frac{\partial y}{\partial s_1}\frac{\partial}{\partial y}+\frac{\partial t}{\partial s_1}\frac{\partial}{\partial t}\\=&\frac{\partial}{\partial s_1}.
	\end{split}\]
	For the same reason \[\Im L=\frac{\partial}{\partial s_2}.\] 
	
	One can see again that $u$ solves \[\frac{\partial u}{\partial s_1}=h_1\] and \[\frac{\partial u}{\partial s_2}=h_2\] with initial condition $u(0,0,s_3)=0$ uniquely. By the inverse function theorem, we can find $u(x,y,t)$ solving $Lu=h$ on $\Gamma$ and such a $u$ is defined on a neighborhood of $\Gamma$ .

\end{proof}

The following extension lemma is classical. One can find it in any book on smooth manifolds, e.g. Lemma 2.26 of \citep{Le13}.

\begin{lemma}[Extension Lemma for Smooth Functions]
	Suppose that $M$ is a smooth manifold with or without boundary, $A\subset M$ is a closed subset, and $f: A\mapsto\mathbb{R}$ is a smooth function. For any open subset $U$ containing $A$, there exists a smooth function $\tilde{f}: M\mapsto\mathbb{R}$ such that $\tilde{f}|_A=f$ and $\mathrm{supp} \tilde{f}\subset U$.
\end{lemma}

\begin{proof}[Proof of Theorem \ref{main}]
	Calculate that
	\[\begin{split}
	\Hessian_{\delta e^\phi}(L, N)=&g(\nabla_L\nabla(\delta e^\phi), N)\\=&
	g(\nabla_L(\delta \nabla e^\phi), N)+g(\nabla_L(e^\phi\nabla\delta), N)\\=&
	\delta g(\nabla_L(\nabla e^\phi), N)+e^\phi L(\delta)\overline{N}(\phi)+e^\phi g(\nabla_L(\nabla\delta), N)+e^\phi L(\phi)\overline{N}(\delta)
	\end{split}\]
	
	On $\partial\Omega$,
	\[
	\Hessian_{\delta e^\phi}(L, N)=e^\phi g(\nabla_L(\nabla\delta), N)+e^\phi L(\phi)\overline{N}(\delta)=e^\phi(\Hessian_\delta(L, N)+L(\phi)\overline{N}(\delta)).\]

	We let $\phi$ be the solution of \[\overline{L}(\phi)=-\frac{\overline{\Hessian_\delta(L, N)}}{N(\delta)},\] from Proposition \ref{cprop} defined on a closed neighborhood $\overline{V}$ of $\Gamma$ in $\partial\Omega$. Then one finds on $\Gamma$ that 
	\[\Hessian_{\delta e^\phi}(L, N)=0.\]
	
	By the extension lemma for smooth functions, we can find $\tilde{\phi}$ defined on a neighborhood $U$ of $\Gamma$ in $\mathbb{C}^2$ such that $\tilde{\phi}\vert_{\overline{V}}=\phi$ . We find a smooth function $\chi$ defined on a neighborhood of $\Omega$ and $\overline{\support\tilde{\phi}}\subset W$ where $W$ is an open subset of $\mathbb{C}^2$ such that $\Gamma\subset\overline{W}\subset U$. Now $\chi$ also satisfies
	\[\chi(z,w)=\begin{cases}
	1, \quad \text{if}(z,w)\in W\\
	0, \quad \text{if}(z,w)\not\in U.
	\end{cases}\]
	Then we define $\Phi$ to be
		\[\Phi(z,w)=\begin{cases}
		\chi\tilde{\phi}, \quad \text{if}(z,w)\in U\\
		0, \quad \text{otherwise}.
		\end{cases}\]
		Then $\delta e^\Phi$ is a defining function of $\Omega$ which satisfies 	\[\Hessian_{\delta e^\Phi}(L, N)=0\] on $\Gamma$. By Theorem \ref{prop}, the theorem is proved.
\end{proof}

Inspired by the preceding proof, the following theorem can be established. 

\begin{theorem}
	Let $\delta$ be an arbitrarily defining function of a bounded domain $\Omega\subset\mathbb{C}^2$ with smooth boundary. Let $\Sigma\subset\partial\Omega$ denote the Levi-flat sets of $\partial\Omega$.  Then \[\Hessian_\rho(L, N)=0\] on $\Sigma$ for the defining function $\rho=\delta e^\phi$ of $\Omega$ if and only if there exists a (real) smooth function $\phi$ so that \[L(\phi)=-\frac{\Hessian_\delta(L, N)}{\overline{N}(\delta)}\quad \text{ on  } \Sigma,\] where $L$ is the normalized holomorphic tangential vector field of $\partial\Omega$ and $N$ is the normalized holomorphic normal vector field of $\partial\Omega$.  Specially, if there is no real function $\phi$ which solves \[L(\phi)=-\frac{\Hessian_\delta(L, N)}{\overline{N}(\delta)},\] then there exists no defining functions $r$ so that $\Hessian_r(L,N)=0$.
\end{theorem}

\begin{proof}
	The first part is true because of the equality: 	\[
	\Hessian_{\delta e^\phi}(L, N)=e^\phi(\Hessian_\delta(L, N)+L(\phi)\overline{N}(\delta)) \quad\text{ on  }\Sigma.\]
	The second part holds because it is well known that every defining function $\rho$ can be written as $\rho=\delta\cdot h$ for some smooth $h>0$. We define $\phi=\log h$ and then the second part follows from the first part.
\end{proof}

From the preceding theorem, combining with Theorem \ref{prop} we have the following result.

\begin{theorem}
	Let $\delta$ be an arbitrarily defining function of a bounded domain $\Omega\subset\mathbb{C}^2$ with smooth boundary. Let $\Sigma\subset\partial\Omega$ denote the Levi-flat sets of $\partial\Omega$. Suppose there is a real function $u$ which solves \[L(u)=-\frac{\Hessian_\delta(L, N)}{\|\nabla\delta\|}\] on $\Sigma$. Then the Diederich-Forn\ae ss index of $\Omega$ is $1$.
\end{theorem}

\section{Infinite Type and Diederich-Forn\ae ss Index}\label{sec5}

In this section, we will answer Question 3 raised in Section \ref{1}. We want to see a new example which has the Diederich--Forn\ae ss index 1 but cannot be verified by formerly known theorems. Our example will neither be of finite type nor admit a plurisubharmonic defining function. Thus, to show the Diederich--Forn\ae ss index to be 1, we have to use our theorems in the current article.

It has been known to the experts that there is no equivalence between finite type and trivial Diederich-Forn\ae ss index.  Nevertheless, one can show that the domain of finite type has Diederich-Forn\ae ss index 1. (Indeed, we could not find a precise reference. Professor Anne-Katrin Gallagher was kind to teach us the following arguments, so the authors owe her the credit.) This is done similar to the usual construction: $re^{-\phi_M}$ where $r$ is some defining function for the domain, $\{\phi_M\}$ are the functions constructed by Catlin in \citep{Ca87}. That is $0\leq\phi_M\leq 1$ and the complex Hessian of $\phi_M$ in a direction $\xi$ is larger than $M|\xi|^2$.

Let us consider a domain of which the Levi-flat points form a real curve transversal to the holomorphic tangent vector fields on the boundary. The reader should be warned that the following domain also admits a defining function which is plurisubharmonic on the boundary.

\begin{example}
	Let \[\Omega=\lbrace (z,w)\in\mathbb{C}^2: |z|^2+2e^{-1/|w|^2}<1\rbrace\] be a bounded domain with smooth boundary in $\mathbb{C}^2$. Moreover, it has an infinite type point at $(1 ,0)$.
\end{example}

We are going to verify the Levi-flat points are only at $(e^{i\theta} ,0)$ for $\theta\in[0.2\pi)$. 
Since \[\rho(z,w)=|z|^2+2e^{-1/|w|^2}-1,\] we obtained that \[\frac{\partial\rho}{\partial w}=2\frac{e^{-1/|w|^2}}{w^2\overline{w}},\]
\[\frac{\partial\rho}{\partial z}=\overline{z},\]
\[\frac{\partial^2\rho}{\partial z\partial{\bar{z}}}=1\]
and  \[\frac{\partial^2\rho}{\partial w\partial{\bar{w}}}=2e^{-1/|w|^2}(\frac{1}{|w|^6}-\frac{1}{|w|^4}).\]

Moreover, the holomorphic tangent vector field is \[L=2\frac{e^{-1/|w|^2}}{w^2\overline{w}}\frac{\partial}{\partial z}-\overline{z}\frac{\partial}{\partial w}\]

and its complex Hessian is 
\[\begin{pmatrix}
1&0\\0&2e^{-1/|w|^2}(\frac{1}{|w|^6}-\frac{1}{|w|^4})
\end{pmatrix}.\]

Thus the Levi form is 
\[\begin{split}
&\begin{pmatrix}
2\frac{e^{-1/|w|^2}}{w^2\overline{w}}&-\overline{z}
\end{pmatrix}\begin{pmatrix}
1&0\\0&2e^{-1/|w|^2}(\frac{1}{|w|^6}-\frac{1}{|w|^4})
\end{pmatrix}\begin{pmatrix}
2\frac{e^{-1/|w|^2}}{w\overline{w}^2}\\-z
\end{pmatrix}\\=&4\frac{e^{-2/|w|^2}}{|w|^6}+2|z|^2e^{-1/|w|^2}\left(\frac{1}{|w|^6}-\frac{1}{|w|^4}\right)\\=&2e^{-1/|w|^2}\left(2\frac{e^{-1/|w|^2}}{|w|^6}+|z|^2\left(\frac{1}{|w|^6}-\frac{1}{|w|^4}\right)\right).
\end{split}\]

We can see that \[\lbrace(z,w)\in\mathbb{C}^2: |z|=1, w=0\rbrace\] is a set of Levi-flat points. To see if all Levi-flat points belong to it, we need to solve the algebraic equation:
\[
	\begin{cases}
	2e^{-1/|w|^2}\left(2\frac{e^{-1/|w|^2}}{|w|^6}+|z|^2\left(\frac{1}{|w|^6}-\frac{1}{|w|^4}\right)\right)=0\\
	|z|^2+2e^{-1/|w|^2}=1.
	\end{cases}
\]
In case $w\neq 0$, the previous equation is equivalent to the following:
\[\begin{split}
0=&2\frac{e^{-1/|w|^2}}{|w|^6}+\left(1-2e^{-1/|w|^2}\right)\left(\frac{1}{|w|^6}-\frac{1}{|w|^4}\right)\\=&\frac{1}{|w|^6}-\frac{1}{|w|^4}+2\frac{e^{-1/|w|^2}}{|w|^4}\\=&\frac{1}{|w|^4}\left(\frac{1}{|w|^2}-1+2e^{-1/|w|^2}\right).
\end{split}\]

To solve \[0=\frac{1}{|w|^2}-1+2e^{-1/|w|^2},\] we let $t=-\dfrac{1}{|w|^2}<0$ and it converts to 
\[0=-t-1+2e^t\] which asserts that $t<0$ has no solution. Thus the complete Levi-flat set is \[\Gamma:=\lbrace(z,w)\in\mathbb{C}^2: |z|=1, w=0\rbrace.\]

The tangent vector on $\Gamma$ is \[x\frac{\partial}{\partial x}+y\frac{\partial}{\partial y}+0\frac{\partial}{\partial u}+0\frac{\partial}{\partial y},\] where $z=x+iy$ and $w=u+iv$. At $\Gamma$, \[L=-\bar{z}\frac{\partial}{\partial w}.\]

Hence the Diederich-Forn\ae ss index of $\Omega$ is 1, because of $L$ is transversal to $\Gamma$. The same computation works for a slightly more general domain as what follows.

\begin{proposition}
	Assume $z_0\in\mathbb{C}$ and $v_0\in\mathbb{R}$. Let $p=(z_0, v_0)$ and \[\Omega_{a,b}(p):=\lbrace(z,w)\in\mathbb{C}^2: a|z-z_0|^2+2e^{-1/|w-v_0|^2}<b\rbrace\] for arbitrary $a>0$ and $b\in(0,2)$. Then the Levi-flat sets $\mathcal{F}_{a,b}(p)$ is \[\lbrace(z, w)\in\mathbb{C}^2: |z-z_0|=\sqrt{\frac{b}{a}}, w=v_0\rbrace, \] is a real curve and is transversal to holomorphic tangent vector fields. Hence, the Diederich-Forn\ae ss index of $\Omega_{a, b}(p)$ is 1.
\end{proposition}
\begin{remark}
	Let $z=x+iy$ and $w=u+iv$ and think $v$-axis as the vertical direction. By calculation, we find out the north pole $\mathcal{N}$ of $\Omega_{a,b}(p)$ is at $\left(0, v_0+\sqrt{\frac{1}{\ln2-\ln b}}\right)$ and $\mathcal{F}_{a,b}(p)$ is the equator. We also define $\Omega^+_{a,b}(p)$ to be $\mathbb{B}(\mathcal{N}, \epsilon)\cap\Omega_{a,b}(p)$, where $\epsilon>0$ is chosen to be very small so that $\mathcal{F}_{a,b}(p)\cap \mathbb{B}(\mathcal{N},2\epsilon)=\varnothing$. We denote the complement of $\Omega^+_{a,b}(p)$ by $\Omega^-_{a,b}(p)$.
\end{remark}

The previous example admits a defining function plurisubharmonic on the boundary. Of course, by Forn\ae ss--Herbig's theorem, the Diederich--Forn\ae ss index is 1. For the following paragraphs, we are going to construct a bounded pseudoconvex, infinite type domain $\tilde{\Omega}$ with smooth boundary which does not admit a defining function which is plurisubharmonic on the boundary. Our example is motivated by the method McNeal uses to prove Proposition 2.1 of \citep{Mc92}. More specifically, we will solder one piece of the domain in \citep{Be85} with our domain $\Omega_{a,b}$ along the strongly pseudoconvex boundary points. The result domain should be infinite type because of $\Omega_{a,b}$ and does not admit a defining function which is plurisubharmonic on the boundary because of Behrens' result. Moreover, the domain has Diederich-Forn\ae ss index 1 by Remark \ref{rmk} because the Levi-flat points are a point united with a real curve transversal to holomorphic tangent vector fields.

Firstly, recall the result of \citep{Be85}.
\begin{theorem}[Behrens]
	There is a $\tau_0>0$ such that $H\cap \mathbb{B}(0, \tau_0)$ is pseudoconvex from the side $\rho<0$ and the origin is the only non-strongly pseudoconvex point, where $H$ is a smooth hypersurface defined by the function
	\[\rho_H(z, w)=v+R(z, w),\]
	where \[\begin{split}
	P_6(z)&=\frac{1}{2}|z|^6+2\Re\left(-\frac{1}{20}\bar{z}^5z+\frac{i}{4}\bar{z}^4z^2\right)\\
	Q_4(z)&=\frac{1}{2}|z|^4-\frac{i}{6}z^3\bar{z}+\frac{i}{6}\bar{z}^3z\\
	R(z, w)&=P_6(z)+2uQ_4(z)+|z|^2u^2+|z|^2u^4+|z|^{10}+|z|^6u^2.
	\end{split}\]
\end{theorem}

Recall that $p=(z_0, v_0)$, and we define \[T(p,(z, w))=a|z-z_0|^2+2e^{-1/|w-v_0|^2}\] for $a>0$.

Let $a>0$ big, $0<\delta<1$ small, $p$ in the side of $\rho_H<0$ so that $\partial\Omega_{a,1}(p)$ intersects $H$ transversally, $H\cap \partial\Omega_{a,1}(p)\subset \mathbb{B}(0,\tau_0)$ and $\Omega^-_{a,b}(p)\subset \lbrace \rho_H<-\delta \rbrace $. By continuity argument, we can also assume there exists $\epsilon_0>0$ so that $H\cap \partial\Omega_{a,1+\epsilon_0}(p)\subset \mathbb{B}(0,\tau_0)$ and $H$ intersects $\partial\Omega_{a, b}(p)$ transversally for any $b\in(1-\epsilon_0,1+\epsilon_0)$.

We define for a fixed $\eta\in (0,1)$
\[\rho=K\chi_1(T(p,(z,w)))+\chi_2(-(-\rho_H(z, w))^\eta),\]
where $K>2$ is to be determined later. Note it here that $-(-\rho_H(z, w))^\eta$ is plurisubharmonic for any $\eta>0$ in $\rho_H<0$ away from origin. This is because origin is the only non-strongly pseudoconvex point. Here we let $\chi_1(t)$ be a real-valued, $C^\infty$ increasing funcion on $\mathbb{R}$ with $\chi_1(t)\equiv 0$ if $t<1-\epsilon_0$ for an $0<\epsilon_0<\frac{1}{2}$ and $\chi_1''(t)>0$ if $t>1-\epsilon_0$ and there is $t_0\in (1-\epsilon_0, 1+\epsilon_0)$ so that $\chi_1(t_0)=t_0$. We also let $\chi_2(t): \mathbb{R}\rightarrow\mathbb{R}$ be a $C^\infty$ increasing function such that $\chi_2(t)\equiv-\delta^\eta$ if $t\leq-(\delta^\eta)$ and $\chi_2(t)=t$ if $t>-\frac{1}{2}\delta^\eta$.

Let \[\tilde{\Omega}:=\lbrace(z, w)\in\mathbb{C}^2: \rho(z,w)<0\rbrace,\] where $K$ is big enough to ensure $\tilde{\Omega}$ being bounded.
We divide $\partial\tilde{\Omega}$ into three sets:
\[\begin{split}
B_1&=\lbrace(z, w)\in\partial\tilde{\Omega}: T(p, (z, w))<1-\epsilon_0 \rbrace\\
B_2&=\lbrace(z, w)\in\partial\tilde{\Omega}: \rho_H(z,w)<-\delta \rbrace\\
B_3&=\lbrace(z, w)\in\partial\tilde{\Omega}: T(p, (z,w))\geq 1-\epsilon_0 \quad\text{and}\quad \rho_H(z,w)>-\delta \rbrace
\end{split}\]

For $B_1$, $\chi_1=0$ so the $\partial\tilde{\Omega}$ is defined by $\rho_H$. For $B_2$, observe that if $T(p, (z,w))=1-\epsilon_0$, \[\rho=-\delta^\eta<0.\] When $(z,w)\in\mathbb{C}^2$ is such that $T(p,(z,w))=t_0\in(1-\epsilon_0,1+\epsilon_0)$,
\[\rho=KT(p, (z,w))-\delta^\eta>2-2\epsilon_0-\delta^\eta>0,\] by shrinking $\delta>0$. Hence $B_2$ is defined by
\[a|z-z_0|^2+2e^{-1/|w-w_0|^2}=\chi^{-1}_1(\frac{\delta^\eta}{K})<2.\] Moreover Levi-flat sets of $B_2$ is a real curve transversal to holomorphic tangent vector fields. For $B_3$, $\rho$ is a plurisubharmonic function, because $T(p, (z,w))$ and $-(-\rho_H)^\eta$ are both plurisubharmonic, for any $\eta$. So $\tilde{\Omega}$ is strongly pseudoconvex on $B_3$. Since 
\[\nabla\rho=K\chi_1'\nabla T(p, (z, w))+\chi_2'(\eta(-\rho_H(z,w))^{\eta-1})\nabla\rho_H,\] and $H$ intersects $\partial\Omega_{a,b}(p)$ transversally for any $b\in(1-\epsilon_0,1+\epsilon_0)$, $\nabla\rho$ is not vanishing.

Hence, we obtain the main theorem in the current section.

\begin{theorem}\label{example}
	$\tilde{\Omega}$ is a bounded pseudoconvex domain in $\mathbb{C}^2$ with smooth boundary which satisfies the following two properties
	\begin{enumerate}
      \item $\tilde{\Omega}$ is neither a domain of finite type nor a domain admitting a defining function which is plurisubharmonic on the boundary. 
      \item The Diederich-Forn\ae ss index of $\tilde{\Omega}$ is 1.
	\end{enumerate}
	 
\end{theorem}

\section*{Appendix}
\subsection*{Proof of Lemma \ref{marco}}

Since $L(\rho)=\overline{L}(\rho)=0$, we have that,
\[	
\begin{split}
&\Hessian_{-(-re^\psi)^\eta e^{-\delta\eta |z|^2}}(aL+bN,aL+bN)\\=&-\eta e^{-\delta\eta |z|^2}(-re^\psi)^{\eta-1}\big(|a|^2\big((\delta^2\eta)(-re^\psi)  L(|z|^2)\overline{L} (|z|^2)-\delta(-re^\psi) \Hessian_{|z|^2}(L, L)-e^\psi L(\psi)\overline{L}(r)\\&-e^\psi L(r)\overline{L}(\psi)-re^\psi L(\psi)\overline{L}(\psi)-e^\psi \Hessian_r(L,L)-re^\psi \Hessian_\psi(L, L)\big)\\&+2\Re \big(a\bar{b} (\delta\eta e^\psi L(|z|^2)\overline{N} (r)+\delta\eta re^\psi L(|z|^2)\overline{N} (\psi)+(\delta^2\eta)(-re^\psi)  L(|z|^2)\overline{N} (|z|^2)-\delta(-re^\psi) \Hessian_{|z|^2}(L, N)\\&-e^\psi L(\psi)\overline{N}(r)-e^\psi L(r)\overline{N}(\psi)-re^\psi L(\psi)\overline{N}(\psi)-e^\psi \Hessian_r(L,N)-re^\psi \Hessian_\psi(L, N)) \big)\\&+|b|^2(\delta\eta N(\rho)\overline{N} (|z|^2)+\delta\eta N(|z|^2)\overline{N} (\rho)+(\delta^2\eta)(-\rho)  N(|z|^2)\overline{N} (|z|^2)-\delta(-\rho) \Hessian_{|z|^2}(N, N)\\&-\Hessian_\rho (N, N)+\frac{1}{-\rho}(\eta-1)N(\rho)\overline{N} (\rho))\big)\\=&-\eta e^{-\delta\eta |z|^2}(-re^\psi)^{\eta-1}\left(|a|^2I+2\Re (a\bar{b}I\!I)+|b|^2I\!I\!I\right).
\end{split}
\]

Recall that $\Hessian_{|z|^2}(L, L)=1$. From these, we can simplify the term $I$ a little bit in $\mathbb{C}^2$ with the following two identities, \[ e^\psi(L\psi)(\overline{L}r)+re^\psi(L\psi)(\overline{L}\psi)=(L\psi)(\overline{L}\rho)=0\quad\text{and}\] \[e^\psi(Lr)(\overline{L}\psi)+re^\psi(L\psi)(\overline{L}\psi)=(\overline{L}\psi)(L\rho)=0.\]

We have that
\[\begin{split}
I=&(\delta^2\eta)(-re^\psi)  L(|z|^2)\overline{L} (|z|^2)-\delta(-re^\psi) \Hessian_{|z|^2}(L, L)-e^\psi L(\psi)\overline{L}(r)-e^\psi L(r)\overline{L}(\psi)\\&-re^\psi L(\psi)\overline{L}(\psi)-e^\psi \Hessian_r(L,L)-re^\psi \Hessian_\psi(L, L)\\=&(\delta^2\eta)(-re^\psi)  L(|z|^2)\overline{L} (|z|^2)-\delta(-re^\psi) \Hessian_{|z|^2}(L, L)+re^\psi L(\psi)\overline{L}(\psi)\\&-e^\psi \Hessian_r(L,L)-re^\psi \Hessian_\psi(L, L)\\=&e^\psi\big((\delta^2\eta)(-r) L(|z|^2)\overline{L} (|z|^2)-\delta(-r) \Hessian_{|z|^2}(L, L)+r L(\psi)\overline{L}(\psi)-\Hessian_r(L,L)\\&-r \Hessian_\psi(L, L)\big)\\=&e^\psi\big((\delta^2\eta)(-r) L(|z|^2)\overline{L} (|z|^2)-\delta(-r)+r L(\psi)\overline{L}(\psi)-\Hessian_r(L,L)-r \Hessian_\psi(L, L)\big),
\end{split}
\]

Since $\Hessian_{|z|^2}(L, N)=0$ we can further simplify $I\!I$ a little:
\[\begin{split}
I\!I=&\delta\eta e^\psi L(|z|^2)\overline{N} (r)+\delta\eta re^\psi L(|z|^2)\overline{N} (\psi)+(\delta^2\eta)(-re^\psi)  L(|z|^2)\overline{N} (|z|^2)-\delta(-re^\psi) \Hessian_{|z|^2}(L, N)\\&-e^\psi L(\psi)\overline{N}(r)-e^\psi L(r)\overline{N}(\psi)-re^\psi L(\psi)\overline{N}(\psi)-e^\psi \Hessian_r(L,N)-re^\psi \Hessian_\psi(L, N)\\=&e^\psi\big(\delta\eta L(|z|^2)\overline{N} (r)+\delta\eta r L(|z|^2)\overline{N} (\psi)+(\delta^2\eta)(-r)  L(|z|^2)\overline{N} (|z|^2)-\delta(-r) \Hessian_{|z|^2}(L, N)\\&- L(\psi)\overline{N}(r)- L(r)\overline{N}(\psi)-r L(\psi)\overline{N}(\psi)- \Hessian_r(L,N)-r \Hessian_\psi(L, N)\big)\\=&e^\psi\big(\delta\eta L(|z|^2)\overline{N} (r)+\delta\eta r L(|z|^2)\overline{N} (\psi)+(\delta^2\eta)(-r)  L(|z|^2)\overline{N} (|z|^2)- L(\psi)\overline{N}(r)- L(r)\overline{N}(\psi)\\&-r L(\psi)\overline{N}(\psi)- \Hessian_r(L,N)-r \Hessian_\psi(L, N)\big).
\end{split}\]

More specifically, on $\partial\Omega$, we have an estimate.
\[|I\!I|<e^\psi\big(\delta\eta |L(|z|^2)\overline{N} (r)|+|L(\psi)\overline{N}(r)|+|\Hessian_r(L, N)|\big).\]

Now, for $I\!I\!I$,	
\[\begin{split}
I\!I\!I=&\delta\eta N(\rho)\overline{N} (|z|^2)+\delta\eta N(|z|^2)\overline{N} (\rho)+(\delta^2\eta)(-\rho)  N(|z|^2)\overline{N} (|z|^2)-\delta(-\rho) \Hessian_{|z|^2}(N, N)\\&-\Hessian_\rho (N, N)+\frac{1}{-\rho}(\eta-1)N(\rho)\overline{N} (\rho)\\=&\delta\eta N(\rho)\overline{N} (|z|^2)+\delta\eta N(|z|^2)\overline{N} (\rho)+(\delta^2\eta)(-\rho)  N(|z|^2)\overline{N} (|z|^2)-\delta(-\rho) \Hessian_{|z|^2}(N, N)\\&-\Hessian_\rho (N, N)+\frac{1}{-re^\psi}(\eta-1)(e^{2\psi}|N(r)|^2+e^{2\psi}r^2|N(\psi)|^2+re^{2\psi} N(r)\overline{N}(\psi)+re^{2\psi} N(\psi)\overline{N}(r))\\=&\delta\eta N(\rho)\overline{N} (|z|^2)+\delta\eta N(|z|^2)\overline{N} (\rho)+(\delta^2\eta)(-\rho)  N(|z|^2)\overline{N} (|z|^2)-\delta(-\rho)-\Hessian_\rho (N, N)\\&+\frac{1}{-re^\psi}(\eta-1)(e^{2\psi}|N(r)|^2+e^{2\psi}r^2|N(\psi)|^2+re^{2\psi} N(r)\overline{N}(\psi)+re^{2\psi} N(\psi)\overline{N}(r))
\end{split}
\]

By the previous equality, after shrinking the neighborhood of $\partial\Omega$ in $\overline{\Omega}$ so that $r$ is sufficiently small, we can also obtain an estimate for $I\!I\!I$:
\[I\!I\!I<\frac{e^\psi}{-2r}(\eta-1)|N(r)|^2.\]

\subsection*{Proof of Lemma \ref{3.1}}
We rewrite the proof with a language of differential geometry. This proof is essentially due to Forn\ae ss--Herbig in \citep{FH07}. Let $\xi=\Hessian_r(N_r, L_r)$ and then $\psi=-C\xi\bar{\xi}$.
We observe that \[L_r(\psi)=-CL_r(\xi\bar{\xi})=-C\xi L_r(\bar{\xi})-C\bar{\xi}L_r(\xi).\]
The preceding equation is $0$ because \[\xi=\Hessian_r(N_r, L_r)=0=\Hessian_r(L_r, N_r)=\bar{\xi}\] on $\Sigma$.

Now we are going to prove (\ref{eqnl}). Observe that on $\partial\Omega$, $L=L_r$ and
\[\begin{split}
\Hessian_{\xi\bar{\xi}}(L_r,L_r)&=g(\nabla_{L_r}\nabla(\xi\bar{\xi}), L_r)=g(\nabla_{L_r}(\xi\nabla\bar{\xi}), L_r)+g(\nabla_{L_r}(\bar{\xi}\nabla\xi), L_r)\\&=2\Re(\xi \Hessian_{\bar{\xi}}(L_r,L_r))+|L_r\xi|^2+|L_r\bar{\xi}|^2
\end{split}\]
which implies \[\Hessian_\psi(L,L)=\Hessian_\psi(L_r,L_r)\leq -C|L_r\Hessian_r(N_r, L_r)|^2\] because $\xi=0$ on $\Sigma$.

Next, we prove \[L_r\Hessian_r(N_r, L_r)=N_r\Hessian_r(L_r, L_r)\] on $\Sigma$. We have
\[\begin{split}
L_r\Hessian_r(N_r, L_r)=&L_rg(\nabla_{N_r}\nabla r, L_r)\\=&g(\nabla_{L_r}\nabla_{N_r}\nabla r, L_r)+g(\nabla_{N_r}\nabla r, \nabla_{\overline{L}_r} L_r)\\=&g(\nabla_{L_r}\nabla_{N_r}\nabla r, L_r)+\Hessian_r(N_r, \nabla_{\overline{L}_r} L_r).
\end{split}\]
Since \[\Hessian_r(L_r, N_r)=\Hessian_r(N_r, L_r)=0\] and \[0=\Hessian_r(L_r, L_r)=L_r(\overline{L}_rr)-(\nabla_{L_r}\overline{L}_r)r=-(\nabla_{L_r}\overline{L}_r)r,\] we have that $\nabla_{\overline{L}_r} L_r$ is proportional to $L_r$ and $\Hessian_r(N_r, \nabla_{\overline{L}_r} L_r)=0$. Hence
\[
L_r\Hessian_r(N_r, L_r)=g(\nabla_{L_r}\nabla_{N_r}\nabla r, L_r)=g(\nabla_{N_r}\nabla_{L_r}\nabla r, L_r)+g(\nabla_{[L_r,N_r]}\nabla r, L_r),
\] because of the vanishing of the sectional curvature of $\mathbb{C}^2$. Therefore
\[\begin{split}
L_r\Hessian_r(N_r, L)=&g(\nabla_{N_r}\nabla_{L_r}\nabla r, L_r)+g(\nabla_{[L_r,N_r]}\nabla r, L_r)\\=&g(\nabla_{N_r}\nabla_{L_r}\nabla r, L_r)+\overline{g(\nabla_{L_r}\nabla r, [L_r,N_r])}\\=&N_rg(\nabla_{L_r}\nabla r, L_r)-g(\nabla_{L_r}\nabla r, \nabla_{\overline{N}_r}L_r)+\overline{g(\nabla_{L_r}\nabla r, [L_r,N_r])},
\end{split}\]
where the second term vanishes because \[0=\Hessian_r(N_r, L_r)=N_r(\overline{L}_rr)-(\nabla_{N_r}\overline{L}_r)r=-(\nabla_{N_r}\overline{L}_r)r\] implies that $\nabla_{\overline{N}_r}L_r$ is proportional to $L_r$ and the third term vanishes because $[L_r, N_r]$ is linearly spanned by $L_r$ and $N_r$ and $g(\nabla_{L_r}\nabla r, L_r)=g(\nabla_{L_r}\nabla r, N_r)=0$. Thus \[L_r\Hessian_r(N_r, L_r)=N_rg(\nabla_{L_r}\nabla r, L_r)=N\Hessian_r(L_r,L_r),\] which completes the proof.

\bigskip
\bigskip

\noindent {\bf Acknowledgments}. We appreciate all kind help we received from Dr. Anne-Katrin Gallagher. The authors also thank to Dr. Bun Wong, Dr. Qi S. Zhang, Xinghong Pan, Dr. Lihan Wang for fruitful conversations. We thanks to Dr. Yuan Yuan for his interests and careful reading. We thank to Jihun Yum for pointing out a typo for us too.

%\bibliographystyle{amsplain}
%\begin{thebibliography}{9}
\printbibliography
%\end{thebibliography}

\nocite{Ko99}
\end{document}